\documentclass[final,leqno]{siamltex704}
\usepackage{epsfig}
\usepackage{amsmath}
\usepackage{amssymb}
\usepackage{tikz}
\usepackage{graphicx}
\usepackage[notcite,notref]{showkeys}
\newtheorem{algorithm}{Weak Galerkin Algorithm}
\newcommand{\bq}{{\bf q}}
\newcommand{\bn}{{\bf n}}
\newcommand{\bx}{{\bf x}}
\newcommand{\bv}{{\bf v}}

\def\T{{\mathcal T}}
\def\E{{\mathcal E}}

\def\l{{\langle}}
\def\r{{\rangle}}

\def\bn{{\bf n}}
\def\bq{{\bf q}}

\def\aa{\mathfrak{a}}
\def\bbQ{\mathbb{Q}}
\newcommand{\pT}{{\partial T}}
\newtheorem{defi}{Definition}[section]
\def\3bar{{|\hspace{-.02in}|\hspace{-.02in}|}}
\renewcommand{\ldots}{\dotsc}
\title{Weak Galerkin Finite Element
Methods on Polytopal Meshes}

\author{Lin Mu\thanks{Department of Mathematics, University of Arkansas at
Little Rock, Little Rock, AR 72204} \and Junping
Wang\thanks{Division of Mathematical Sciences, National Science
Foundation, Arlington, VA 22230 (jwang@\break nsf.gov). The research
of Wang was supported by the NSF IR/D program, while working at
National Science Foundation. However, any opinion, finding, and
conclusions or recommendations expressed in this material are those
of the author and do not necessarily reflect the views of the
National Science Foundation,} \and Xiu Ye\thanks{Department of
Mathematics, University of Arkansas at Little Rock, Little Rock, AR
72204 (xxye@ualr.edu). This research was supported in part by
National Science Foundation Grant DMS-1115097.}}

\begin{document}

\maketitle

\begin{abstract}
This paper introduces a new weak Galerkin (WG) finite element method
for second order elliptic equations on polytopal meshes. This
method, called WG-FEM, is designed by using a discrete weak gradient
operator applied to discontinuous piecewise polynomials on finite
element partitions of arbitrary polytopes with certain shape
regularity. The paper explains how the numerical schemes are
designed and why they provide reliable numerical approximations for
the underlying partial differential equations. In particular,
optimal order error estimates are established for the corresponding
WG-FEM approximations in both a discrete $H^1$ norm and the standard
$L^2$ norm. Numerical results are presented to demonstrate the
robustness, reliability, and accuracy of the WG-FEM. All the results
are derived for finite element partitions with polytopes. Allowing
the use of discontinuous approximating functions on arbitrary
polytopal elements is a highly demanded feature for numerical
algorithms in scientific computing.
\end{abstract}

\begin{keywords}
weak Galerkin, finite element methods, discrete gradient, second-order
elliptic problems, polyhedral meshes
\end{keywords}

\begin{AMS}
Primary: 65N15, 65N30; Secondary: 35J50
\end{AMS}
\pagestyle{myheadings}

\section{Introduction}\label{Section:Introduction}
In this paper, we are concerned with a further and new development
of weak Galerkin (WG) finite element methods for partial
differential equations. Our model problem is a second-order elliptic
equation which seeks an unknown function $u=u(x)$ satisfying
\begin{equation}\label{pde}
-\nabla\cdot(a(x,u,\nabla u)\nabla u)=f(x), \quad \mbox{in}\ \Omega,
\end{equation}
where $\Omega$ is a polytopal domain in $\mathbb{R}^d$ (polygonal or
polyhedral domain for $d=2,3$), $\nabla u$ denotes the gradient of
the function $u=u(x)$, and $a=a(x,u,\nabla u)$ is a symmetric
$d\times d$ matrix-valued function in $\Omega$. We shall assume that
the differential operator is strictly elliptic in $\Omega$; that is,
there exists a positive number $\lambda>0$ such that
\begin{equation}\label{ellipticity}
\xi^ta(x,\eta,p)\xi\ge \lambda \xi^t\xi,\qquad\forall
\xi\in\mathbb{R}^d,
\end{equation}
for all $x\in \Omega, \eta\in \mathbb{R}, p\in \mathbb{R}^d$. Here
$\xi$ is understood as a column vector and $\xi^t$ is the transpose
of $\xi$. We also assume that the differential operator has bounded
coefficients; that is for some constant $\Lambda$ we have
\begin{equation}\label{boundedness}
 |a(x,\eta,p)|\leq \Lambda,
 \end{equation}
for all $x\in \Omega, \eta\in \mathbb{R},$ and $p\in \mathbb{R}^d$.

Introduce the following form
\begin{equation}\label{bilinear-form}
\mathfrak{a}(\phi;u,v):=\int_\Omega a(x,\phi,\nabla \phi)\nabla
u\cdot\nabla v dx.
\end{equation}
For simplicity, let the function $f$ in (\ref{pde}) be locally
integrable in $\Omega$. We shall consider solutions of (\ref{pde})
with a non-homogeneous Dirichlet boundary condition
\begin{equation}\label{bc}
u=g, \quad \mbox{on}\  \partial\Omega,
\end{equation}
where $g\in H^{\frac12}(\partial\Omega)$ is a function defined on
the boundary of $\Omega$. Here $H^1(\Omega)$ is the Sobolev space
consisting of functions which, together with their gradients, are
square integrable over $\Omega$. $H^{\frac12}(\partial\Omega)$ is
the trace of $H^1(\Omega)$ on the boundary of $\Omega$. The
corresponding weak form seeks $u\in H^1(\Omega)$ such that $u=g$ on
$\partial\Omega$ and
\begin{eqnarray}\label{weak-form}
\mathfrak{a}(u;u,v)=F(v),\qquad \forall v\in H_0^1(\Omega),
\end{eqnarray}
where $F(v)\equiv \int_\Omega f v dx$.

Galerkin finite element methods for (\ref{weak-form}) refer to
numerical techniques that seek approximate solutions from a finite
dimensional space $V_h$ consisting of piecewise polynomials on a
prescribed finite element partition $\T_h$. The method is called
conforming if $V_h$ is a subspace of $H^1(\Omega)$. Conforming
finite element methods are then formulated by solving $u_h\in V_h$
such that $u_h=I_hg$ on $\partial\Omega$ and
\begin{eqnarray}\label{w1-conforming}
\mathfrak{a}(u_h;u_h,v)=F(v),\qquad \forall v\in V_h\cap
H_0^1(\Omega),
\end{eqnarray}
where $I_hg$ is a certain approximation of the Dirichlet boundary
value. When $V_h$ is not a subspace of $H^1(\Omega)$, the form
$\mathfrak{a}(\phi;u,v)$ is no longer meaningful since the gradient
operator is not well-defined for non-$H^1$ functions in the
classical sense. Nonconforming finite element methods arrive when
the gradients in $\mathfrak{a}(\phi;u,v)$ are taken locally on each
element where the finite element functions are polynomials. More
precisely, the form $\mathfrak{a}(\phi;u,v)$ in nonconforming finite
element methods is given element-by-element as follows
\begin{equation}\label{nc-form}
\mathfrak{a}_h(\phi;u,v):=\sum_{T\in \T_h}\int_T a(x,\phi,\nabla
\phi)\nabla u\cdot\nabla v dx.
\end{equation}
When $V_h$ is close to be conforming, the form
$\mathfrak{a}_h(\phi;u,v)$ shall be an acceptable approximation to
the original form $\mathfrak{a}(\phi;u,v)$. The key in the
nonconforming method is to explore the maximum non-conformity of
$V_h$ when the approximate form $\mathfrak{a}_h(\phi;u,v)$ is
required to be sufficiently close to the original form.

A natural generalization of the nonconforming finite element method
would occur when the following extended form of (\ref{nc-form}) is
employed
\begin{equation}\label{wg-form}
\mathfrak{a}_w(\phi;u,v):=\sum_{T\in \T_h}\int_T a(x,\phi,\nabla_w
\phi)\nabla_w u\cdot\nabla_w v dx,
\end{equation}
where $\nabla_w$ is an approximation of $\nabla$ locally on each
element. By viewing $\nabla_w$ as a weakly defined gradient
operator, the form $\mathfrak{a}_w(\phi;u,v)$ would give a new class
of numerical methods called {\em weak Galerkin (WG)} finite element
methods.

In general, weak Galerkin refers to finite element techniques for
partial differential equations in which differential operators
(e.g., gradient, divergence, curl, Laplacian) are approximated by
weak forms as distributions. In \cite{wy}, a WG method was
introduced and analyzed for second order elliptic equations based on
a {\em discrete weak gradient} arising from local {\em RT} \cite{rt}
or {\em BDM} \cite{bdm} elements. Due to the use of the {\em RT} and
{\em BDM} elements, the WG finite element formulation of \cite{wy}
was limited to classical finite element partitions of triangles
($d=2$) or tetrahedra ($d=3$). In \cite{wy1}, a weak Galerkin finite
element method was developed for the second order elliptic equation
in the mixed form. The use of a stabilization for the flux variable
in the mixed formulation is the key to the WG mixed finite element
method of \cite{wy1}. The resulting WG mixed finite element schemes
turned out to be applicable for general finite element partitions
consisting of shape regular polytopes (e.g., polygons in 2D and
polyhedra in 3D), and the stabilization idea opened a new door for
weak Galerkin methods.

The goal of this paper is to apply the stabilization idea to the
form $\mathfrak{a}_w(\phi;u,v)$, and thus to develop a new weak
Galerkin method for (\ref{pde})-(\ref{bc}) in the primary variable
$u$ that shall admit general finite element partitions consisting of
arbitrary polytopal elements. The resulting WG method will no longer
be limited to {\em RT} and {\em BDM} elements in the computation of
the discrete weak gradient $\nabla_w$. In practice, allowing
arbitrary shape in finite element partition provides a convenient
flexibility in both numerical approximation and mesh generation,
especially in regions where the domain geometry is complex. Such a
flexibility is also very much appreciated in adaptive mesh
refinement methods.

The main contribution of this paper is three fold: (1) the WG finite
element method to be described in section \ref{Section:wg-fem}
allows finite element partitions of arbitrary polytopes which are
shape regular in the sense as defined in \cite{wy1}, (2) the finite
element spaces constitute regular polynomial spaces on each
element/face which are computation-friendly, and (3) the WG finite
element scheme retains the mass conservation property of the
original system locally on each element.

One close relative of the WG finite element method of this paper is
the hybridizable discontinuous Galerkin (HDG) method \cite{cgl}. In
fact, it can be proved that our weak Galerkin method is identical to
HDG method for the Poisson equation. However, the WG method differs
from HDG for the model problem (\ref{pde}) with nonconstant
coefficient matrix $a$ and more sophisticated problems. These two
methods are fundamentally different in concept and formulation. The
key element of HDG is the flux variable, while the key element for
WG is the gradient operator through weak derivatives. For either
nonlinear or degenerate coefficient matrix $a=a(x,u,\nabla u)$, the
WG finite element method has obvious advantage over HDG since
$\nabla u$ is approximated by $\nabla_w u$ and there is no need to
invert the matrix $a$ in WG formulations. More importantly, the
concept of weak derivatives makes WG a widely applicable numerical
technique for a large variety of partial differential equations
which we shall report in forthcoming papers.

The paper is organized as follows. In section
\ref{Section:preliminaries}, we introduce some standard notations in
Sobolev spaces. In section \ref{Section:weak-gradient}, we review
the definition and approximation of the weak gradient operator.  In
section \ref{Section:wg-fem}, we provide a detailed description for
the new WG finite element scheme, including a discussion on the
element shape regularity assumption. In section
\ref{Section:L2projections}, we define some local projection
operators and then derive some approximation properties which are
useful in error analysis. In section
\ref{Section:wg-massconservation}, we show that the WG finite
element method retains the mass conservation property of the
original system locally on each element. In section
\ref{Section:existence}, we show that the weak Galerkin finite
element scheme for the nonlinear problem has at least one solution.
The solution existence is based on the Leray-Schauder fixed point
theorem. In section \ref{Section:error-analysis}, we shall establish
an optimal order error estimate for the WG finite element
approximation in a $H^1$-equivalent discrete norm for the linear
case of (\ref{pde}). We shall also derive an optimal order error
estimate in the $L^2$ norm by using a duality argument as was
commonly employed in the standard Galerkin finite element methods
\cite{ci, sue}. Finally in section
\ref{Section:numerical-experiments}, we present some numerical
results which confirm the theory developed in earlier sections.

\section{Preliminaries and Notations}\label{Section:preliminaries}

Let $D$ be any domain in $\mathbb{R}^d, d=2, 3$. We use the standard
definition for the Sobolev space $H^s(D)$ and their associated inner
products $(\cdot,\cdot)_{s,D}$, norms $\|\cdot\|_{s,D}$, and
seminorms $|\cdot|_{s,D}$ for any $s\ge 0$. For example, for any
integer $s\ge 0$, the seminorm $|\cdot|_{s, D}$ is given by
$$
|v|_{s, D} = \left( \sum_{|\alpha|=s} \int_D |\partial^\alpha v|^2
dD \right)^{\frac12}
$$
with the usual notation
$$
\alpha=(\alpha_1, \ldots, \alpha_d), \quad |\alpha| =
\alpha_1+\ldots+\alpha_d,\quad
\partial^\alpha =\prod_{j=1}^d\partial_{x_j}^{\alpha_j}.
$$
The Sobolev norm $\|\cdot\|_{m,D}$ is given by
$$
\|v\|_{m, D} = \left(\sum_{j=0}^m |v|^2_{j,D} \right)^{\frac12}.
$$

The space $H^0(D)$ coincides with $L^2(D)$, for which the norm and
the inner product are denoted by $\|\cdot \|_{D}$ and
$(\cdot,\cdot)_{D}$, respectively. When $D=\Omega$, we shall drop
the subscript $D$ in the norm and inner product notation.

The space $H({\rm div};D)$ is defined as the set of vector-valued
functions on $D$ which, together with their divergence, are square
integrable; i.e.,
\[
H({\rm div}; D)=\left\{ \bv: \ \bv\in [L^2(D)]^d, \nabla\cdot\bv \in
L^2(D)\right\}.
\]
The norm in $H({\rm div}; D)$ is defined by
$$
\|\bv\|_{H({\rm div}; D)} = \left( \|\bv\|_{D}^2 + \|\nabla
\cdot\bv\|_{D}^2\right)^{\frac12}.
$$

\section{Weak Gradient}\label{Section:weak-gradient}

The key in weak Galerkin methods is the use of discrete weak
derivatives in the place of strong derivatives in the variational
form for the underlying partial differential equations. For the
model problem (\ref{weak-form}), the gradient $\nabla$ is the
principle differential operator involved in the variational
formulation. Thus, it is critical to define and understand discrete
weak gradients for the corresponding numerical methods. Following
the idea originated in \cite{wy}, the discrete weak gradient is
given by approximating the weak gradient operator with piecewise
polynomial functions; details are presented in the rest of this
section.

Let $K$ be any polytopal domain with boundary $\partial K$. A {\em
weak function} on the region $K$ refers to a function $v=\{v_0,
v_b\}$ such that $v_0\in L^2(K)$ and $v_b\in H^{\frac12}(\partial
K)$. The first component $v_0$ can be understood as the value of $v$
in $K$, and the second component $v_b$ represents $v$ on the
boundary of $K$. Note that $v_b$ may not necessarily be related to
the trace of $v_0$ on $\partial K$ should a trace be well-defined.
Denote by $W(K)$ the space of weak functions on $K$; i.e.,
\begin{equation}\label{hi.888}
W(K):= \{v=\{v_0, v_b \}:\ v_0\in L^2(K),\; v_b\in
H^{\frac12}(\partial K)\}.
\end{equation}
The weak gradient operator, as was introduced in \cite{wy}, is
defined as follows.
\medskip

\begin{defi}
The dual of $L^2(K)$ can be identified with itself by using the
standard $L^2$ inner product as the action of linear functionals.
With a similar interpretation, for any $v\in W(K)$, the {\em weak
gradient} of $v$ is defined as a linear functional $\nabla_w v$ in
the dual space of $H(div,K)$ whose action on each $q\in H(div,K)$ is
given by
\begin{equation}\label{weak-gradient}
(\nabla_w v, q)_K := -(v_0, \nabla\cdot q)_K + \langle v_b,
q\cdot\bn\rangle_{\partial K},
\end{equation}
where $\bn$ is the outward normal direction to $\partial K$,
$(v_0,\nabla\cdot q)_K=\int_K v_0 (\nabla\cdot q)dK$ is the action
of $v_0$ on $\nabla\cdot q$, and $\langle v_b,
q\cdot\bn\rangle_{\partial K}$ is the action of $q\cdot\bn$ on
$v_b\in H^{\frac12}(\partial K)$.
\end{defi}

\medskip

The Sobolev space $H^1(K)$ can be embedded into the space $W(K)$ by
an inclusion map $i_W: \ H^1(K)\to W(K)$ defined as follows
$$
i_W(\phi) = \{\phi|_{K}, \phi|_{\partial K}\},\qquad \phi\in H^1(K).
$$
With the help of the inclusion map $i_W$, the Sobolev space $H^1(K)$
can be viewed as a subspace of $W(K)$ by identifying each $\phi\in
H^1(K)$ with $i_W(\phi)$. Analogously, a weak function
$v=\{v_0,v_b\}\in W(K)$ is said to be in $H^1(K)$ if it can be
identified with a function $\phi\in H^1(K)$ through the above
inclusion map. It is not hard to see that the weak gradient is
identical with the strong gradient (i.e., $\nabla_w v=\nabla v$) for
smooth functions $v\in H^1(K)$.

\medskip
Recall that the discrete weak gradient operator was defined by
approximating $\nabla_w$ in a polynomial subspace of the dual of
$H(div,K)$. More precisely, for any non-negative integer $r\ge 0$,
denote by $P_{r}(K)$ the set of polynomials on $K$ with degree no
more than $r$. The discrete weak gradient operator, denoted by
$\nabla_{w,r, K}$, is defined as the unique polynomial
$(\nabla_{w,r, K}v) \in [P_r(K)]^d$ satisfying the following
equation
\begin{equation}\label{discrete-weak-divergence-element}
(\nabla_{w,r, K}v, q)_K = -(v_0,\nabla\cdot q)_K+ \langle v_b,
q\cdot\bn\rangle_{\partial K},\qquad \forall q\in [P_r(K)]^d.
\end{equation}

The discrete weak gradient operator, namely $\nabla_{w,r, K}$ as
defined in (\ref{discrete-weak-divergence-element}), was first
introduced in \cite{wy} where two examples of the polynomial
subspace $[P_r(K)]^d$ were thoroughly discussed and employed for the
second order elliptic problem (\ref{pde})-(\ref{bc}). The two
examples make use of the Raviat-Thomas \cite{rt} and
Brezzi-Douglas-Marini \cite{bdm} elements developed in the classical
mixed finite element method. As a result, the corresponding WG
finite element method of \cite{wy} is closely related to the mixed
finite element method. In this paper, we shall allow a greater
flexibility in the definition and computation of the discrete weak
gradient operator $\nabla_{w,r, K}$ by using the usual polynomial
space $[P_r(K)]^d$. This will result in a new class of WG finite
element schemes with remarkable properties to be detailed in forth
coming sections.

\section{Weak Galerkin Finite Element Schemes}\label{Section:wg-fem}

In finite element methods, mesh generation is a crucial first step
in the algorithm design. For the usual finite element methods
\cite{ci, bf}, the meshes are mostly required to be simplices:
triangles or quadrilaterals in two dimensions and tetrahedra or
hexahedra in three dimensions, or their variations known as
isoparametric elements. Our new weak Galerkin finite element method
is designed to be sufficiently flexible so that general meshes of
polytopes (e.g., polygons in 2D and polyhedra in 3D) are allowed.
For simplicity, we shall refer the elements as polygons or polyhedra
in the rest of the paper.

\subsection{Domain Partition}
Let ${\cal T}_h$ be a partition of the domain $\Omega$ consisting of
polygons in two dimensions or polyhedra in three dimensions
satisfying a set of conditions to be specified. Denote by ${\cal
E}_h$ the set of all edges or flat faces in ${\cal T}_h$, and let
${\cal E}_h^0={\cal E}_h\backslash\partial\Omega$ be the set of all
interior edges or flat faces. For every element $T\in \T_h$, we
denote by $|T|$ the area or volume of $T$ and by $h_T$ its diameter.
Similarly, we denote by $|e|$ the length or area of $e$ and by $h_e$
the diameter of edge or flat face $e\in\E_h$. We also set as usual
the mesh size of $\T_h$ by
$$
h=\max_{T\in\T_h} h_T.
$$
All the elements of $\T_h$ are assumed to be closed and simply
connected polygons or polyhedra. We need some shape regularity for
the partition $\T_h$ described as follows (see \cite{wy1} for more
details).

\medskip
\begin{description}
\item[A1:] \ Assume that there exist two positive constants $\varrho_v$ and $\varrho_e$
such that for every element $T\in\T_h$ we have
\begin{equation}\label{a1}
\varrho_v h_T^d\leq |T|,\qquad \varrho_e h_e^{d-1}\leq |e|
\end{equation}
for all edges or flat faces $e$ of $T$.

\item[A2:] \ Assume that there exists a positive constant $\kappa$ such that for every element
$T\in\T_h$ we have
\begin{equation}\label{a2}
\kappa h_T\leq h_e
\end{equation}
for all edges or flat faces $e$ of $T$.

\item[A3:] \ Assume that the mesh edges or faces are flat. We
further assume that for every $T\in\T_h$, and for every edge/face
$e\in \partial T$, there exists a pyramid $P(e,T, A_e)$ contained in
$T$ such that its base is identical with $e$, its apex is $A_e\in
T$, and its height is proportional to $h_T$ with a proportionality
constant $\sigma_e$ bounded from below by a fixed positive number
$\sigma^*$. In other words, the height of the pyramid is given by
$\sigma_e h_T$ such that $\sigma_e\ge \sigma^*>0$. The pyramid is
also assumed to stand up above the base $e$ in the sense that the
angle between the vector $\bx_e-A_e$, for any $\bx_e\in e$, and the
outward normal direction of $e$ is strictly acute by falling into an
interval $[0, \theta_0]$ with $\theta_0< \frac{\pi}{2}$.

\item[A4:] \ Assume that each $T\in\T_h$ has a circumscribed simplex $S(T)$ that is
shape regular and has a diameter $h_{S(T)}$ proportional to the
diameter of $T$; i.e., $h_{S(T)}\leq \gamma_* h_T$ with a constant
$\gamma_*$ independent of $T$. Furthermore, assume that each
circumscribed simplex $S(T)$ intersects with only a fixed and small
number of such simplices for all other elements $T\in\T_h$.
\end{description}

\medskip
Figure \ref{fig:shape-regular-element} is a depiction of a
shape-regular polygonal element in 2D. As to the property {\bf A3},
for edge $e=AF$, the corresponding pyramid is given by the triangle
$AA_eF$ which is of a similar size as the polygonal element. $\bn_e$
is the outward normal direction to the edge $e$. The angle between
the two vectors $\bn_e$ and $\overrightarrow{A_e\bx_e}$ is strictly
acute for any $\bx_e\in e$.

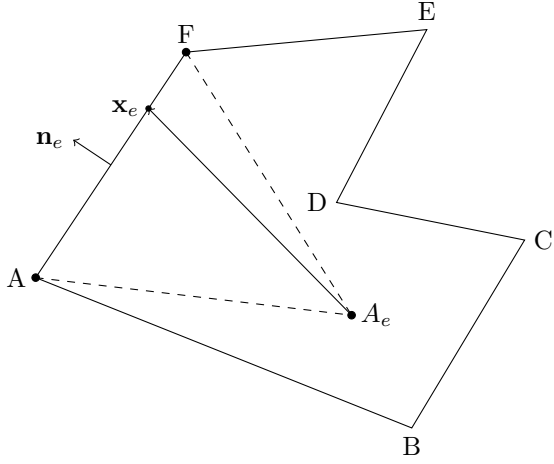
\begin{figure}[h!]
\begin{center}
\begin{tikzpicture}
\coordinate (A) at (-3,0); \coordinate (B) at (2,-2); \coordinate
(C) at (3.5, 0.5); \coordinate (D) at (1,1); \coordinate (E) at
(2.2, 3.3); \coordinate (F) at (-1, 3); \coordinate (CC) at (0,0);
\coordinate (Ae) at (1.2,-0.5); \coordinate (xe) at (-1.5, 2.25);
\coordinate (AFc) at (-2, 1.5); \coordinate (AFcLeft) at (-2.5,
1.83333); \draw node[left] at (xe) {$\bx_e$}; \draw node[right] at
(Ae) {$A_e$}; \draw node[left] at (A) {A}; \draw node[below] at (B)
{B}; \draw node[right] at (C) {C}; \draw node[left] at (D) {D};
\draw node[above] at (E) {E}; \draw node[above] at (F) {F}; \draw
node[left] at (AFcLeft) {$\bn_e$}; \draw
(A)--(B)--(C)--(D)--(E)--(F)--cycle; \draw[dashed](A)--(Ae);
\draw[dashed](F)--(Ae); \draw[->] (Ae)--(xe); \draw[->]
(AFc)--(AFcLeft); \filldraw[black] (A) circle(0.05);
    \filldraw[black] (F) circle(0.05);
    \filldraw[black] (Ae) circle(0.05);
    \filldraw[black] (xe) circle(0.035);
\end{tikzpicture}
\caption{Depiction of a shape-regular polygonal element $ABCDEFA$.}
\label{fig:shape-regular-element}
\end{center}
\end{figure}

\subsection{WG Finite Element Algorithms}
Let ${\cal T}_h$ be a finite element partition that is shape
regular; namely, satisfying the properties {\bf A1-A4}. On each
element $T\in\T_h$, we have a space of weak functions W(T) defined
as in Section \ref{Section:weak-gradient}. Denote by $V$ the weak
function space on $\T_h$ given by
\begin{equation}\label{vspace}
 V:=\{v=\{v_0, v_b\}:\ \{v_0, v_b\}|_T\in W(T),
T\in \T_h \},
\end{equation}
where $\{v_0, v_b\}|_T:=\{(v_0)|_T, (v_b)|_{\partial T}\}$ is the
restriction of $v$ on the element $T$.

For any given integer $k\ge 1$, let $W_k(T)$ be the discrete weak
function space consisting of polynomials of degree $k$ in $T$ and
piecewise polynomials of degree $k$ on $\partial T$; i.e.,
\begin{equation}\label{Wkspace}
W_k(T):=\{v=\{v_0,v_b\}:\; {v_0}|_T\in P_k(T), \ v_b|_e\in P_k(e), \
e\in\partial T\}.
\end{equation}
Furthermore, let $V_h$ be the weak Galerkin finite element space
defined as follows
\begin{equation}\label{vhspace}
V_h:=\{v=\{v_0,v_b\}:\; \{v_0,v_b\}|_T\in W_k(T),\ T\in \T_h\}
\end{equation}
and
\begin{equation}\label{vh0space}
V^0_h:=\{v: \ v\in V_h,\  v_b=0 \mbox{ on } \partial\Omega\}.
\end{equation}

Denote by $\nabla_{w,k-1}$ the discrete weak gradient operator on
the finite element space $V_h$ computed by using
(\ref{discrete-weak-divergence-element}) on each element $T$; i.e.,
$$
(\nabla_{w,k-1}v)|_T =\nabla_{w,k-1, T} (v|_T),\qquad \forall v\in
V_h.
$$
For simplicity of notation, from now on we shall drop the subscript
$k-1$ in the notation $\nabla_{w,k-1}$ for the discrete weak
gradient.

Now we introduce two forms on $V_h$ as follows:
\begin{eqnarray*}
\aa(\phi;v,\;w) & = & \sum_{T\in {\cal T}_h}\int_T a(x,\phi,\nabla_w\phi)\nabla_w v\cdot\nabla_w w dT,\\
s(v,\;w) & = & \rho\sum_{T\in {\cal T}_h} h_T^{-1}\langle
v_0-v_b,\;\;w_0-w_b\rangle_{\partial T},
\end{eqnarray*}
where $\rho>0$ is a parameter with constant value. In practical
computation, one might set $\rho=1$. Denote by
$\aa_s(\cdot;\cdot,\;\cdot)$ a stabilization of
$\aa(\cdot;\cdot,\;\cdot)$ given by
$$
\aa_s(\phi;v,\;w):=\aa(\phi;v,\;w)+s(v,\;w).
$$

\begin{algorithm}
A numerical approximation for (\ref{pde}) and (\ref{bc}) can be
obtained by seeking $u_h=\{u_0,\;u_b\}\in V_h$ satisfying both $u_b=
Q_b g$ on $\partial \Omega$ and the following equation:
\begin{equation}\label{wg}
\aa_s(u_h;u_h,\;v)=(f,\;v_0), \quad\forall\ v=\{v_0,\; v_b\}\in
V_h^0,
\end{equation}
where $Q_b g$ is an approximation of the Dirichlet boundary value in
the polynomial space $P_k(\partial T\cap \partial\Omega)$. For
simplicity, one may take $Q_b g$ as the standard $L^2$ projection of
the boundary value $g$ on each boundary segment.
\end{algorithm}

\section{$L^2$ Projection Operators}\label{Section:L2projections}

There are two basic polynomial spaces associated with each element
$T\in \T_h$. The first one is the local finite element space
$W_k(T)$ and the second one is the polynomial space $[P_{k-1}(T)]^d$
which was utilized to define the discrete weak gradient $\nabla_w$
in (\ref{discrete-weak-divergence-element}); namely, the operator
$\nabla_{w,r,K}$ with $r=k-1$ and $K=T$. For simplicity of
discussion, we introduce the following notation
$$
G_{k-1}(T):=[P_{k-1}(T)]^d,
$$
and shall call this a local discrete gradient space.

For each element $T\in \T_h$, denote by $Q_{0}$ the $L^2$ projection
from $L^2(T)$ onto $P_k(T)$. Analogously, for each edge or flat face
$e\in {\cal E}_h$, let $Q_{b}$ be the $L^2$ projection operator from
$L^2(e)$ onto $P_k(e)$. Denote by $\bbQ_h$ the $L^2$ projection onto
the local discrete gradient space $G_{k-1}(T)$. Recall that $V$ is
the weak function space as defined by (\ref{vspace}). We define a
projection operator $Q_h: V \to V_h$ as follows
\begin{equation}\label{Qh-def}
Q_h v:=\{Q_0v_0, Q_bv_b\},\qquad\forall\ v=\{v_0,v_b\}\in V.
\end{equation}

\begin{lemma}
Let $Q_h$ be the projection operator defined as in (\ref{Qh-def}).
Then, on each element $T\in\T_h$, we have
\begin{equation}\label{key}
\nabla_w (Q_h \phi) = \bbQ_h (\nabla \phi),\quad\forall \phi\in
H^1(\Omega).
\end{equation}
\end{lemma}
\begin{proof}
Using (\ref{discrete-weak-divergence-element}), the integration by
parts and the definitions of $Q_h$ and $\bbQ_h$, we have that for
any $\tau\in G_{k-1}(T)$
\begin{eqnarray*}
(\nabla_w (Q_h \phi),\; \tau)_T &=& -(Q_0 \phi,\; \nabla\cdot\tau)_T
+\langle Q_b \phi,\; \tau\cdot\bn\rangle_{\pT}\\
&=&-(\phi,\; \nabla\cdot\tau)_T + \langle \phi,\; \tau\cdot\bn\rangle_{\partial T}\\
&=&(\nabla \phi,\; \tau)_T=(\bbQ_h(\nabla\phi),\; \tau)_T
\end{eqnarray*}
which implies the desired relation (\ref{key}).
\end{proof}

\medskip
The following lemma provides some estimate for the projection
operators $Q_h$ and $\bbQ_h$. Observe that the underlying mesh
$\T_h$ is assumed to be sufficiently general to allow polygons or
polyhedra. A proof of the lemma can be found in \cite{wy1}. It
should be pointed out that the proof of the lemma requires some
non-trivial technical tools in analysis, which have also been
established in \cite{wy1}.

\begin{lemma}
Let $\T_h$ be a finite element partition of $\Omega$ satisfying the
shape regularity assumption {\bf A1 - A4}.  Then, for any $\phi\in
H^{k+1}(\Omega)$, we have
\begin{eqnarray}
&&\sum_{T\in\T_h} \|\phi-Q_0\phi\|_{T}^2 +\sum_{T\in\T_h}h_T^2
\|\nabla(\phi-Q_0\phi)\|_{T}^2\le C h^{2(k+1)}
\|\phi\|^2_{k+1},\label{Qh}\\
&&\sum_{T\in\T_h} \|a(\nabla\phi-\bbQ_h(\nabla\phi))\|^2_{T} \le
Ch^{2k} \|\phi\|^2_{k+1}.\label{Rh}
\end{eqnarray}
Here and in what follows of this paper, $C$ denotes a generic
constant independent of the meshsize $h$ and the functions in the
estimates.
\end{lemma}

\medskip
Let $T$ be an element with $e$ as an edge.  For any function
$\varphi\in H^1(T)$, the following trace inequality has been proved
to be valid for general meshes satisfying {\bf A1 - A4} (see
\cite{wy1} for details):
\begin{equation}\label{trace}
\|\varphi\|_{e}^2 \leq C \left( h_T^{-1} \|\varphi\|_T^2 + h_T
\|\nabla \varphi\|_{T}^2\right).
\end{equation}
Using (\ref{trace}), we can obtain the following estimates.

\begin{lemma} Assume that $\T_h$ is shape regular. Then for any $w\in H^{k+1}(\Omega)$ and
$v=\{v_0,v_b\}\in V_h$, we have
\begin{eqnarray}
\left|\sum_{T\in\T_h} h_T^{-1}\langle Q_0w-Q_bw,\; v_0-v_b\rangle_\pT\right|&\le&
Ch^k\|w\|_{k+1}\3bar v\3bar,\label{mmm1}\\
\left|\sum_{T\in\T_h} \langle a(\nabla w-\bbQ_h\nabla w)\cdot\bn,\;
v_0-v_b\rangle_\pT\right| &\leq& C h^k\|w\|_{k+1} \3bar
v\3bar.\label{mmm2}
\end{eqnarray}
\end{lemma}

\begin{proof}
Using the definition of $Q_h$, (\ref{trace}), and (\ref{Qh}), we
have
\begin{eqnarray*}
&&\left|\sum_{T\in\T_h} h_T^{-1}\langle Q_0w-Q_bw,\; v_0-v_b\rangle_\pT\right|
= \left|\sum_{T\in\T_h} h_T^{-1} \langle Q_0w-w,\; v_0-v_b\rangle_\pT\right|\\
&&\le C\left(\sum_{T\in\T_h}(h_T^{-2}\|Q_0w-w\|_T^2+
\|\nabla (Q_0w-w)\|_T^2)\right)^{1/2}\left(\sum_{T\in\T_h}h_T^{-1}\|v_0-v_b\|^2_{\pT}\right)^{1/2}\\
&&\le Ch^k\|w\|_{k+1}\3bar v\3bar.
\end{eqnarray*}
Similarly, it follows from (\ref{trace}) and (\ref{Rh}) that
\begin{eqnarray*}
&&\left|\sum_{T\in\T_h}\langle a(\nabla w-\bbQ_h\nabla w)\cdot\bn,\;
v_0-v_b\rangle_\pT\right|\\
&&\le \left(\sum_{T\in\T_h}h_T\|a(\nabla w-\bbQ_h\nabla
w)\|_{\pT}^2\right)^{1/2}
\left(\sum_{T\in\T_h}h_T^{-1}\|v_0-v_b\|^2_{\pT}\right)^{1/2}\\
&&\le Ch^k\|w\|_{k+1}\3bar v\3bar.
\end{eqnarray*}
This completes the proof.
\end{proof}

\section{On Mass Conservation}\label{Section:wg-massconservation}

The second order elliptic equation (\ref{pde}) can be rewritten in a
conservative form as follows:
$$
\nabla \cdot q = f, \quad q=-a\nabla u.
$$
Let $T$ be any control volume. Integrating the first equation over
$T$ yields the following integral form of mass conservation:
\begin{equation}\label{conservation.01}
\int_{\partial T} q\cdot \bn ds = \int_T f dT.
\end{equation}
We claim that the numerical approximation from the weak Galerkin
finite element method (\ref{wg}) for (\ref{pde}) retains the mass
conservation property (\ref{conservation.01}) with an appropriately
defined numerical flux $q_h$. To this end, for any given $T\in {\cal
T}_h$, we chose in (\ref{wg}) a test function $v=\{v_0, v_b=0\}$ so
that $v_0=1$ on $T$ and $v_0=0$ elsewhere. It follows from
(\ref{wg}) that
\begin{equation}\label{mass-conserve.08}
\int_T a\nabla_{w} u_h\cdot \nabla_{w}v dT +\rho
h_T^{-1}\int_{\partial T} (u_0-u_b)ds = \int_T f dT.
\end{equation}
Recall that $\bbQ_h$ is the local $L^2$ projection onto
$[P_{k-1}(T)]^d$. Using the definition
(\ref{discrete-weak-divergence-element}) for $\nabla_{w}v$ one
arrives at
\begin{eqnarray}
\int_T a\nabla_{w} u_h\cdot \nabla_{w}v dT &=& \int_T
\bbQ_h(a\nabla_{w} u_h)\cdot \nabla_{w}v dT \nonumber\\
&=& - \int_T \nabla\cdot \bbQ_h(a\nabla_{w} u_h) dT \nonumber\\
&=& - \int_{\partial T} \bbQ_h(a\nabla_{w}u_h)\cdot\bn ds.
\label{conserv.88}
\end{eqnarray}
Substituting (\ref{conserv.88}) into (\ref{mass-conserve.08}) yields
\begin{equation}\label{mass-conserve.09}
\int_{\partial T} \left\{-\bbQ_h\left(a\nabla_{w}u_h\right)+\rho
h_T^{-1}(u_0-u_b)\bn\right\}\cdot\bn ds = \int_T f dT,
\end{equation}
which indicates that the weak Galerkin method conserves mass with a
numerical flux given by
$$
q_h = - \bbQ_h\left(a\nabla_{w}u_h\right)+\rho h_T^{-1}(u_0-u_b)\bn.
$$

Next, we verify that the normal component of the numerical flux,
namely $q_h\cdot\bn$, is continuous across the boundary of each
element $T$. To this end, let $e$ be an interior edge/face shared by
two elements $T_1$ and $T_2$. Choose a test function $v=\{v_0,v_b\}$
so that $v_0\equiv 0$ and $v_b=0$ everywhere except on $e$. It
follows from (\ref{wg}) that
\begin{eqnarray}\label{mass-conserve.108}
\int_{T_1\cup T_2} a\nabla_{w} u_h\cdot \nabla_{w}v dT & & -\rho
h_{T_1}^{-1}\int_{\partial T_1\cap e} (u_0-u_b)|_{T_1}v_bds \\
& &- \rho h_{T_2}^{-1}\int_{\partial T_2\cap e}
(u_0-u_b)|_{T_2}v_bds\nonumber\\
& & =0.\nonumber
\end{eqnarray}
Using the definition of weak gradient
(\ref{discrete-weak-divergence-element}) we obtain
\begin{eqnarray*}
\int_{T_1\cup T_2} a\nabla_{w} u_h\cdot \nabla_{w}v dT&=&
\int_{T_1\cup T_2} \bbQ_h(a\nabla_{w} u_h)\cdot \nabla_{w}v dT\\
&=& \int_e\left(\bbQ_h(a\nabla_{w} u_h)|_{T_1}\cdot\bn_1 +
\bbQ_h(a\nabla_{w} u_h)|_{T_2}\cdot\bn_2\right)v_b ds,
\end{eqnarray*}
where $\bn_i$ is the outward normal direction of $T_i$ on the edge
$e$. It is clear that $\bn_1+\bn_2=0$. Substituting the above
equation into (\ref{mass-conserve.108}) yields
\begin{eqnarray*}
\int_e\left(-\bbQ_h(a\nabla_{w} u_h)|_{T_1}+\rho
h_{T_1}^{-1}(u_0-u_b)|_{T_1}\bn_1\right)\cdot\bn_1 v_bds\\
=-\int_e \left(-\bbQ_h(a\nabla_{w} u_h)|_{T_2}+\rho h_{T_2}^{-1}
(u_0-u_b)|_{T_2}\bn_2\right)\cdot\bn_2 v_bds,
\end{eqnarray*}
which shows the continuity of the numerical flux $q_h$ in the normal
direction.

\section{Existence and Boundedness of WG
Solutions}\label{Section:existence}

Let $\phi\in V_h$ be any weak finite element function. A linearized
version of (\ref{wg}) seeks $u_h=\{u_0,\;u_b\}\in V_h$ satisfying
both $u_b= Q_b g$ on $\partial \Omega$ and the following equation:
\begin{equation}\label{wg-linearized}
\aa_s(\phi;u_h,\;v)=(f,\;v_0), \quad\forall\ v=\{v_0,\; v_b\}\in
V_h^0.
\end{equation}
It is easy to see that, for any fixed $\phi\in V_h$, the bilinear
form $\aa_s(\phi;\cdot,\cdot)$ is symmetric and positive definite in
the weak finite element space $V_h$. Thus, one may introduce a norm
in $V_h$ as follows
\begin{equation}\label{3barnorm}
\3bar v\3bar_\phi:=\sqrt{\aa_s(\phi;v,\;v)},\qquad v\in V_h.
\end{equation}
The assumptions (\ref{ellipticity}) and (\ref{boundedness}) on the
matrix coefficient $a=a(x,\eta,p)$ imply that the norm $\3bar
\cdot\3bar_\phi$ are uniformly equivalent for all $\phi$. In
particular, we shall use the norm arising from $\phi=0$ and denote
the corresponding norm by
$$
\3bar \cdot\3bar:=\3bar \cdot\3bar_0.
$$
The trip-bar norm $\3bar \cdot\3bar$ is an $H^1$-equivalence for
finite element functions with vanishing boundary value. Moreover,
the following Poincar\'{e}-type inequality holds true for functions
in $V_{h}^0$.

\medskip
\begin{lemma}
Assume that the finite element partition $\T_h$ is shape regular.
Then, there exists a constant $C$ independent of the meshsize $h$
such that
\begin{eqnarray}\label{eq:discretepoincare1}
\| v_0 \| &\leq C \3bar v\3bar, \qquad \forall \ v=\{v_0,v_b\}\in
V_{h}^0.
\end{eqnarray}
\end{lemma}

\begin{proof}
For any $v=\{v_0,v_b\}\in V_{h}^0$, let $\bq\in [H^1(\Omega)]^d$ be
such that $\nabla\cdot\bq = v_0$ and $\|\bq\|_1 \leq C \|v_0\|$. To
see an existence of such a function $\bq$, one may first extend
$v_0$ by zero to a convex domain $\tilde\Omega$ which contains
$\Omega$, and then consider the Poisson equation $\Delta \Psi=v_0$
on the enlarged domain $\tilde \Omega$ and set $\bq=\nabla\Psi$. The
required properties of $\bq$ follow immediately from the full
regularity of the Poisson equation on convex domains.

Recall that $\bbQ_h$ is the $L^2$ projection to the space of
piecewise polynomials of degree $k-1$. Thus,
\begin{eqnarray}\nonumber
\|v_0\|^2&=& \sum_{T\in\T_h}(v_0,\nabla\cdot\bq)_T\\
&=& \sum_{T\in\T_h}\left( \langle v_0,\bq\cdot\bn \rangle_{\partial
T} -
(\nabla v_0, \bq)_T\right)\nonumber\\
&=& \sum_{T\in\T_h}\left( \langle v_0,\bq\cdot\bn \rangle_{\partial
T} -
(\nabla v_0, \bbQ_h\bq)_T\right)\nonumber \\
&=& \sum_{T\in\T_h}\left( (v_0,\nabla\cdot(\bbQ_h\bq))_T-\langle
v_0,(\bbQ_h\bq)\cdot\bn \rangle_{\partial T} +\langle
v_0,\bq\cdot\bn
\rangle_{\partial T}\right)\nonumber\\
&=& \sum_{T\in\T_h}\left( (v_0,\nabla\cdot(\bbQ_h\bq))_T-\langle
v_0,(\bbQ_h\bq)\cdot\bn \rangle_{\partial T} +\langle
v_0-v_b,\bq\cdot\bn \rangle_{\partial
T}\right),\label{eq:poincare-1}
\end{eqnarray}
where we have used the continuity of $\bq\cdot\bn$ across each
element edge/face and the fact that $v_b=0$ on $\partial\Omega$.
Observe that the definition (\ref{discrete-weak-divergence-element})
of the discrete weak gradient implies
$$
(v_0,\nabla\cdot(\bbQ_h\bq))_T=-(\nabla_w v, \bbQ_h\bq)_T+\langle
v_b,(\bbQ_h\bq)\cdot\bn \rangle_{\partial T}.
$$
Substituting the above identity into (\ref{eq:poincare-1}) yields
\begin{eqnarray}
\|v_0\|^2&=&\sum_{T\in\T_h}\left(- (\nabla_w v,\bbQ_h\bq)_T +\langle
v_0-v_b,(\bq-\bbQ_h\bq)\cdot\bn \rangle_{\partial
T}\right).\label{eq:poincare-2}
\end{eqnarray}
Using (\ref{mmm2}) with $w=\Psi, a=1$, and $k=1$, we have
$$
\left| \sum_{T\in\T_h} \langle v_0-v_b,(\bq-\bbQ_h\bq)\cdot\bn
\rangle_{\partial T} \right| \leq C h \|\Psi\|_2 \3bar v\3bar.
$$
Substituting the above estimate into (\ref{eq:poincare-2}) we arrive
at
\begin{eqnarray*}
\|v_0\|^2 &\leq& \|\nabla_w v\| \ \|\bbQ_h\bq\| + Ch\|\Psi\|_2 \3bar
v\3bar\\
&\leq&\|\nabla_w v\| \ \|\bq\| + Ch\|\Psi\|_2 \3bar
v\3bar\\
&\leq& C \3bar\nabla_w v\3bar \ \|\Psi\|_2\\
&\leq& C \3bar \nabla_w v\3bar \ \|v_0\|,
\end{eqnarray*}
where we have used the fact that $\bq=\nabla\Psi$ and
$\|\Psi\|_2\leq C \|v_0\|$ for some constant $C$. This completes the
proof of the lemma.
\end{proof}
\medskip

Denote by $B_{g,h}$ the set of finite element functions satisfying
the boundary condition $Q_b g$; i.e.,
$$
B_{g,h}:=\{v=\{v_0, v_b\}\in V_h \mbox{ such that $v_b=Q_b g$ on
$\partial\Omega$}\}.
$$
The following is a result on the solution uniqueness and existence
for the linearized problem (\ref{wg-linearized}).

\begin{lemma}
The weak Galerkin finite element scheme (\ref{wg-linearized}) has
one and only one solution. Moreover, there exists a constant $C$
such that the solution of (\ref{wg-linearized}) has the following
boundedness estimate
\begin{equation}\label{desired_estimate}
\3bar u_h\3bar_\phi \leq C(\|f\|+\inf_{\psi\in
B_{g,h}}\3bar\psi\3bar).
\end{equation}
\end{lemma}

\smallskip

\begin{proof}
It suffices to show that the solution of (\ref{wg-linearized}) is
trivial if the data is homogenous; i.e., if $f=g=0$. To this end,
assume that the data is homogeneous. By taking $v=u_h$ in
(\ref{wg-linearized}) we arrive at
\[
(\tilde a \nabla_w u_h,\;\nabla_w u_h)+\rho\sum_{T\in\T_h}
h_T^{-1}\langle u_0-u_b,\; u_0-u_b\rangle_\pT=0,
\]
where $\tilde a = a(x, \phi,\nabla_w\phi)$. This implies that
$\nabla_w u_h=0$ on each element $T$ and $u_0=u_b$ on $\pT$. It
follows from $\nabla_w u_h=0$ and
(\ref{discrete-weak-divergence-element}) that for any $q\in
[P_{k-1}(T)]^d$ we have
\begin{eqnarray*}
0&=&(\nabla_w u_h,\;q)_T\\
&=&-(u_0,\;\nabla\cdot q)_T+\langle u_b,\;q\cdot\bn\rangle_\pT\\
&=&(\nabla u_0,\; q)_T-\langle u_0-u_b,\;
q\cdot\bn\rangle_\pT\\
&=&(\nabla u_0,\; q)_T.
\end{eqnarray*}
Letting $q=\nabla u_0$ in the above equation yields $\nabla u_0=0$
on $T\in {\cal T}_h$. It follows that $u_0=const$ on any $T\in\T_h$.
This, together with the fact that $u_0=u_b$ on $\partial T$ and
$u_b=0$ on $\partial\Omega$, implies $u_0=u_b=0$.

For any $\psi\in B_{g,h}$, the difference $\tilde u_h=u_h-\psi$ is a
function in $V_h^0$ satisfying the following equation
\begin{eqnarray*}
\aa_s(\phi;\tilde u_h,\;v)=(f,\;v_0)-\aa_s(\phi;\psi,\;v),
\quad\forall\ v=\{v_0,\; v_b\}\in V_h^0.
\end{eqnarray*}
By letting $v=\tilde u_h$ we arrive at
\begin{eqnarray*}
\3bar\tilde u_h\3bar_\phi^2=(f,\;\tilde u_h)-\aa_s(\phi;\psi,\tilde
u_h).
\end{eqnarray*}
Thus, it follows from the Poincar\'e inequality
(\ref{eq:discretepoincare1}) and the boundedness of
$\aa_s(\phi;\cdot,\cdot)$ that
$$
\3bar\tilde u_h\3bar_\phi^2 \leq C(\|f\|+\3bar\psi\3bar) \
\3bar\tilde u_h\3bar_\phi,
$$
which, together with $u_h=\tilde u_h+\psi$ and the usual triangle
inequality, implies the designed estimate (\ref{desired_estimate}).
This completes the proof of the lemma.
\end{proof}
\medskip

For the general nonlinear elliptic equation (\ref{pde}), we have the
following result on solution existence.
\begin{lemma}
There exists a weak finite element function $u_h\in V_h$ satisfying
the weak Galerkin finite element scheme (\ref{wg}). Moreover, the WG
solution satisfies the following estimate:
\begin{equation}\label{wg_estimate}
\3bar u_h\3bar_\phi \leq C(\|f\|+\inf_{\psi\in
B_{g,h}}\3bar\psi\3bar).
\end{equation}
\end{lemma}
\smallskip

\begin{proof}
We shall use the Leray-Schauder fixed point theorem to prove an
existence of $u_h$ satisfying (\ref{wg}). Recall that one version of
the Leray-Schauder fixed point theorem (see for example Theorem 11.3
in \cite{gilbarg}) asserts that a continuous mapping $F$ in
$\mathbb{R}^n$ into itself has at least one fixed point if there
exists a constant $M$ such that any solution of $\sigma F(w) = w$
with $\sigma\in [0,1]$ must satisfy $\|w\|_{{\mathbb{R}^n}} < M$,
where $\|w\|_{{\mathbb{R}^n}}$ is the norm of $w$ in $\mathbb{R}^n$.

For any $\phi\in V_h$, let $u_\phi\in V_h$ be the solution of the
following linear problem: Find $u_\phi=\{u_0,\;u_b\}\in V_h$
satisfying both $u_b= Q_b g$ on $\partial \Omega$ and the following
equation:
\begin{equation}\label{wg-linearized-inproof}
\aa_s(\phi;u_\phi,\;v)=(f,\;v_0), \quad\forall\ v=\{v_0,\; v_b\}\in
V_h^0.
\end{equation}
Denote by $F(\phi):=u_\phi$ the mapping from $V_h$ into itself. It
is clear that $F$ is a continuous one. Assume that $\xi_h\in V_h$
satisfies the operator equation $\xi_h=\sigma F(\xi_h)$ for some
real number $\sigma\in [0,1]$. This implies that $\xi_h=\sigma Q_bg$
on $\partial\Omega$ and satisfies
\begin{equation}\label{wg-linearized-psi}
\aa_s(\xi_h;\sigma^{-1}\xi_h,\;v)=(f,\;v_0), \quad\forall\
v=\{v_0,\; v_b\}\in V_h^0.
\end{equation}
Multiplying both sides of (\ref{wg-linearized-psi}) by $\sigma$
yields
\begin{equation}\label{wg-linearized-psi-new}
\aa_s(\xi_h;\xi_h,\;v)=(\sigma f,\;v_0), \quad\forall\ v=\{v_0,\;
v_b\}\in V_h^0.
\end{equation}
The estimate (\ref{desired_estimate}) can be used to give the
following estimate for the solution of (\ref{wg-linearized-psi-new})
\begin{equation}\nonumber
\3bar\xi_h\3bar_\phi \leq C\sup_{\sigma\in
[0,1]}(\sigma\|f\|+\inf_{\psi\in B_{\sigma g,h}}\3bar\psi\3bar).
\end{equation}
This shows that all the conditions of the Leray-Schauder fixed point
theorem are satisfied for the mapping $F$. Thus, $F$ admits at least
one fixed point $u_h$ which is easily seen to be the solution of the
WG finite element scheme (\ref{wg}).
\end{proof}

\section{Error Analysis}\label{Section:error-analysis}
The goal of this section is to establish some error estimates for
the WG finite element solution $u_h$ arising from (\ref{wg}). Our
convergence analysis will be established for only the linear case of
(\ref{pde}). In other words, we shall assume that the coefficient
matrix $a=a(x,\eta,p)$ is independent of the variables $\eta$ and
$p$. The error will be measured in two natural norms: the triple-bar
norm as defined in (\ref{3barnorm}) and the standard $L^2$ norm. The
triple bar norm is essentially a discrete $H^1$ norm for the
underlying weak function.

For simplicity of analysis, we assume that the coefficient tensor
$a$ in (\ref{pde}) is a piecewise constant matrix with respect to
the finite element partition $\T_h$. The result can be extended to
variable tensors without any difficulty, provided that the tensor
$a$ is piecewise sufficiently smooth.

\subsection{Error equation} Let $\phi\in H^1(T)$ and $v\in V_h$ be
any finite element function. It follows from (\ref{key}), the
definition of the discrete weak gradient
(\ref{discrete-weak-divergence-element}), and the integration by
parts that
\begin{eqnarray}
(a\nabla_w Q_h\phi,\;\nabla_w v)_T&=&(a \bbQ_h(\nabla\phi),\;\nabla_w v)_T\nonumber\\
&=& -(v_0,\ \nabla\cdot (a \bbQ_h\nabla\phi))_T+\langle v_b,\ (a \bbQ_h\nabla\phi)\cdot\bn\rangle_\pT\nonumber\\
&=&(\nabla v_0,\; a\bbQ_h\nabla\phi)_T-\langle v_0-v_b,\ (a\bbQ_h\nabla\phi)\cdot\bn\rangle_\pT\nonumber\\
&=&(a\nabla\phi,\;\nabla v_0)_T-\l (a\bbQ_h\nabla\phi)\cdot\bn,\
v_0-v_b\r_\pT.\label{j1}
\end{eqnarray}
Testing (\ref{pde}) by using $v_0$ of $v=\{v_0,\;v_b\}\in V_h^0$ we
arrive at
\begin{equation}\label{m1}
\sum_{T\in\T_h}(a\nabla u,\;\nabla v_0)_T-\sum_{T\in\T_h} \langle a\nabla u\cdot\bn,\; v_0-v_b\rangle_\pT=(f,\; v_0),
\end{equation}
where we have used the fact that $\sum_{T\in\T_h}\langle a\nabla
u\cdot\bn,\; v_b\rangle_\pT=0$. By letting $\phi=u$ in (\ref{j1}),
we have from combining (\ref{j1}) and (\ref{m1}) that
\[
\sum_{T\in\T_h} (a\nabla_w Q_hu,\;\nabla_w v)_T=(f,\;v_0)+
\sum_{T\in\T_h}\langle a(\nabla u-\bbQ_h\nabla u)\cdot\bn,\;
v_0-v_b\rangle_\pT.
\]
Adding $s(Q_hu,\ v)$ to both sides of the above equation gives
\begin{equation}\label{j2}
a_s(Q_hu,\ v)=(f,\;v_0)+ \sum_{T\in\T_h} \langle a(\nabla
u-\bbQ_h\nabla u)\cdot\bn,\; v_0-v_b\rangle_\pT +s(Q_hu,\ v).
\end{equation}
Subtracting (\ref{wg}) from (\ref{j2}) yields the following error
equation
\begin{eqnarray}
a_s(e_h,\ v)=\sum_{T\in\T_h} \langle a(\nabla u-\bbQ_h\nabla
u)\cdot\bn,\; v_0-v_b\rangle_\pT+ s(Q_hu,\ v),\quad \forall v\in
V_h^0,\label{ee}
\end{eqnarray}
where
$$
e_h=\{e_0,\;e_b\}:=\{Q_0u-u_0,\;Q_bu-u_b\}
$$
is the error between the WG finite element solution and the $L^2$
projection of the exact solution.

\subsection{Error estimates}
The error equation (\ref{ee}) can be used to derive the following
error estimate for the WG finite element solution.

\begin{theorem} Let $u_h\in V_h$ be the weak Galerkin finite element solution of the problem
(\ref{pde})-(\ref{bc}) arising from (\ref{wg}). Assume that the
exact solution is so regular that $u\in H^{k+1}(\Omega)$. Then,
there exists a constant $C$ such that
\begin{equation}\label{err1}
\3bar u_h-Q_hu\3bar \le Ch^{k}\|u\|_{k+1}.
\end{equation}
\end{theorem}
\begin{proof}
By letting $v=e_h$ in (\ref{ee}), we have
\begin{eqnarray}
\3bar e_h\3bar^2&=&\sum_{T\in\T_h} \langle a(\nabla u-\bbQ_h\nabla
u)\cdot\bn,\; e_0-e_b\rangle_\pT+s(Q_hu,\;\ e_h).\label{main}
\end{eqnarray}
It then follows from (\ref{mmm1}) and (\ref{mmm2}) that
\[
\3bar e_h\3bar^2 \le Ch^k\|u\|_{k+1}\3bar e_h\3bar,
\]
which implies (\ref{err1}). This completes the proof.
\end{proof}

\medskip
To obtain an error estimate in the standard $L^2$ norm, we consider
a dual problem that seeks $\Phi\in H_0^1(\Omega)$ satisfying
\begin{eqnarray}
-\nabla\cdot (a \nabla\Phi)&=& e_0\quad
\mbox{in}\;\Omega.\label{dual}
\end{eqnarray}
Assume that the usual $H^{2}$-regularity is satisfied for the dual
problem. This means that there exists a constant $C$ such that
\begin{equation}\label{reg}
\|\Phi\|_2\le C\|e_0\|.
\end{equation}

\begin{theorem} Let $u_h\in V_h$ be the weak Galerkin finite element solution of the problem
(\ref{pde})-(\ref{bc}) arising from (\ref{wg}). Assume that the
exact solution is so regular that $u\in H^{k+1}(\Omega)$. In
addition, assume that the dual problem (\ref{dual}) has the usual
$H^2$-regularity. Then, there exists a constant $C$ such that
\begin{equation}\label{err2}
\|Q_0u-u_0\| \le Ch^{k+1}\|u\|_{k+1}.
\end{equation}
\end{theorem}

\begin{proof}
By testing (\ref{dual}) with $e_0$ we obtain
\begin{eqnarray}\nonumber
\|e_0\|^2&=&-(\nabla\cdot (a\nabla\Phi),e_0)\\
&=&\sum_{T\in\T_h}(a\nabla \Phi,\ \nabla e_0)_T-\sum_{T\in\T_h}\l
a\nabla\Phi\cdot\bn,\ e_0- e_b\r_{\pT}.\label{jw.08}
\end{eqnarray}
Setting $\phi=\Phi$ and $v=e_h$ in (\ref{j1}) yields
\begin{eqnarray}
(a\nabla_w Q_h\Phi,\;\nabla_w e_h)_T=(a\nabla\Phi,\;\nabla e_0)_T-\l
(a\bbQ_h\nabla\Phi)\cdot\bn,\ e_0-e_b\r_\pT.\label{j1-new}
\end{eqnarray}
Substituting (\ref{j1-new}) into (\ref{jw.08}) gives
\begin{equation}\label{m2}
\|e_0\|^2=(a\nabla_w e_h,\ \nabla_w Q_h\Phi)+\sum_{T\in\T_h}\l
a(\bbQ_h\nabla\Phi-\nabla\Phi)\cdot\bn,\ e_0-e_b\r_{\pT}.
\end{equation}
It follows from the error equation (\ref{ee}) that
\begin{eqnarray}
(a\nabla_w e_h,\ \nabla_w Q_h\Phi)&=&\sum_{T\in\T_h} \langle
a(\nabla u-\bbQ_h\nabla u)\cdot\bn,\;
Q_0\Phi-Q_b\Phi\rangle_\pT\nonumber\\
&+&s(Q_hu,\ Q_h\Phi)-s(e_h,\ Q_h\Phi).\label{m3}
\end{eqnarray}
By combining (\ref{m2}) with (\ref{m3}) we arrive at
\begin{eqnarray}\label{m4}
\|e_0\|^2&=&\sum_{T\in \T_h} \langle a(\nabla u-\bbQ_h\nabla
u)\cdot\bn,\;
Q_0\Phi-Q_b\Phi\rangle_\pT\\
&&+s(Q_hu,\ Q_h\Phi)-s(e_h,\ Q_h\Phi)\nonumber\\
&&+\sum_{T\in\T_h}\l a(\bbQ_h\nabla\Phi-\nabla\Phi)\cdot\bn,\
e_0-e_b\r_{\pT}.\nonumber
\end{eqnarray}

Let us bound the terms on the right hand side of (\ref{m4}) one by
one. Using the Cauchy-Schwarz inequality and the definition of $Q_b$
we obtain
\begin{eqnarray}\label{1st-term}
&&\left|\sum_{T\in\T_h} \langle a(\nabla u-\bbQ_h\nabla
u)\cdot\bn,\;
Q_0\Phi-Q_b\Phi\rangle_\pT \right|\\
&& \le \left(\sum_{T\in\T_h}\|a(\nabla u-\bbQ_h\nabla
u)\|^2_\pT\right)^{1/2}
\left(\sum_{T\in\T_h}\|Q_0\Phi-Q_b\Phi\|^2_\pT\right)^{1/2}\nonumber \\
&&\le C\left(\sum_{T\in\T_h}\|a(\nabla u-\bbQ_h\nabla
u)\|^2_\pT\right)^{1/2}
\left(\sum_{T\in\T_h}\|Q_0\Phi-\Phi\|^2_\pT\right)^{1/2} \nonumber
\end{eqnarray}
From the trace inequality (\ref{trace}) and the estimate (\ref{Qh})
we have
$$
\left(\sum_{T\in\T_h}\|Q_0\Phi-\Phi\|^2_\pT\right)^{1/2} \leq C
h^{\frac32}\|\Phi\|_2
$$
and
$$
\left(\sum_{T\in\T_h}\|a(\nabla u-\bbQ_h\nabla
u)\|^2_\pT\right)^{1/2} \leq Ch^{k-\frac12}\|u\|_{k+1}.
$$
Substituting the above two inequalities into (\ref{1st-term}) we
obtain
\begin{eqnarray}\label{1st-term-complete}
\left|\sum_{T\in\T_h} \langle a(\nabla u-\bbQ_h\nabla u)\cdot\bn,\;
Q_0\Phi-Q_b\Phi\rangle_\pT \right| \leq C h^{k+1} \|u\|_{k+1}
\|\Phi\|_2.
\end{eqnarray}
Analogously, it follows from the definition of $Q_b$, the trace
inequality (\ref{trace}), and the estimate (\ref{Qh}) that
\begin{eqnarray}\nonumber
\left|s(Q_hu,\; Q_h\Phi)\right|&\le & \rho\sum_{T\in\T_h}h_T^{-1}
\left|(Q_0u-Q_bu,\ Q_0\Phi-Q_b\Phi)_\pT\right|\\
&\le& C\left(\sum_{T\in\T_h}h_T^{-1}\|Q_0u-u\|^2_\pT\right)^{1/2}
\left(\sum_{T\in\T_h}h_T^{-1}\|Q_0\Phi-\Phi\|^2_\pT\right)^{1/2}\nonumber  \\
&\le& Ch^{k+1}\|u\|_{k+1}\|\Phi\|_2.\label{2nd-term-complete}
\end{eqnarray}
The estimates (\ref{mmm1}) and (\ref{err1}) imply
\begin{eqnarray}\label{3rd-term-complete}
|s(e_h,\ Q_h\Phi)|\le Ch\|\Phi\|_2\3bar e_h\3bar\le
Ch^{k+1}\|u\|_{k+1}\|\Phi\|_2.
\end{eqnarray}
Similarly, it follows from (\ref{mmm2}) and (\ref{err1}) that
\begin{eqnarray}\label{4th-term-complete}
\left|\sum_{T\in\T_h}\l a(\bbQ_h\nabla \Phi-\nabla \Phi)\cdot\bn,\
e_0-e_b\r_\pT\right| &\le&  Ch^{k+1}\|u\|_{k+1}\|\Phi\|_2.
\end{eqnarray}
Now substituting (\ref{1st-term-complete})-(\ref{4th-term-complete})
into (\ref{m4}) yields
$$
\|e_0\|^2 \leq C h^{k+1}\|u\|_{k+1} \|\Phi\|_2,
$$
which, combined with the regularity assumption (\ref{reg}), gives
the desired optimal order error estimate (\ref{err2}).
\end{proof}

\section{Numerical Experiments}\label{Section:numerical-experiments}
The goal of this section is to numerically verify the convergence
theory for the WG finite element method (\ref{wg}) through some
computational examples. In particular, the following issues shall be
examined:
\begin{enumerate}
\item[(N1)] rate of convergence for WG solutions in various measures;
\item[(N2)] accuracy of WG solutions on polyhedral meshes with and without
hanging nodes.
\end{enumerate}
For simplicity, all the numerical experiments are conducted by using
piecewise linear functions (i.e., $k=1$) in the finite element space
$V_h$ as defined in (\ref{vhspace}).

\medskip
For any given $v=\{v_0,v_b\}\in V_h$, recall that its discrete weak
gradient, $\nabla_w v\in [P_0(T)]^d$, is defined locally by the
following equation
$$
(\nabla_w v,{\bf q})_T=-(v_0,\nabla\cdot{\bf q})_T+\langle v_b,{\bf
q}\cdot{\bf n}\rangle_{\partial T},\qquad \forall {\bf q}\in
[P_0(T)]^d.
$$
Since ${\bf q}\in [P_0(T)]^d$, the above equation can be simplified
as
\begin{eqnarray}
(\nabla_w v,{\bf q})_T=\langle v_b,{\bf q}\cdot{\bf
n}\rangle_{\partial T}.
\end{eqnarray}

The error for the WG solution of (\ref{wg}) shall be measured in
three norms defined as follows:
\begin{eqnarray*}
\3bar v \3bar^2:=\sum_{T\in\mathcal{T}_h}\bigg(\int_T|\nabla_w
v|^2dx+h_{T}^{-1}\int_{\partial T}(v_0-v_b)^2ds\bigg)&&\qquad\mbox{(A discrete $H^1$-norm)},\\
\|v_h\|^2:=\sum_{T\in\mathcal{T}_h}\int_T|v_0|^2dx&&\qquad\mbox{(Element-based
$L^2$-norm)},\\
\|v\|_{\mathcal{E}_h}^2:=\sum_{e\in\mathcal{E}_h}h_e\int_e
|v_b|^2ds&&\qquad\mbox{(Edge-based $L^2$-norm)}.
\end{eqnarray*}

\subsection{Case 1: Poisson Problem on Uniform Meshes} Consider the Poisson
problem that seeks an unknown function $u=u(x,y)$ satisfying
$$
-\Delta u=f
$$
in the square domain $\Omega=(0,1)^2$ with homogeneous Dirichlet
boundary condition. The exact solution is given by $u=\sin(\pi
x)\sin(\pi y)$, and the function $f=f(x,y)$ is given to match the
exact solution.

\begin{table}[h]
\caption{Case 1. WG solutions and their convergence on rectangular
elements.}\label{ex1_rect} \center
\begin{tabular}{|c||c|c|c|}
\hline
meshsize $h^{-1}$ & $\3bar Q_hu-u_h\3bar$ & $\|Q_hu-u_h\|$ & $\|Q_hu-u_h\|_{\mathcal{E}_h}$\\
\hline\hline
  4    &7.8668e-001  &1.3782e-001 &1.7244e-02 \\ \hline
  8    &3.6731e-001  &3.5717e-002 &4.5321e-03\\ \hline
  16   &1.7954e-001  &9.0101e-003 &1.1362e-03\\ \hline
  32   &8.9221e-002  &2.2576e-003 &2.8401e-04\\ \hline
  64   &4.4541e-002  &5.6472e-004 &7.0995e-05\\ \hline
  128  &2.2262e-002  &1.4120e-004 &1.7748e-05\\ \hline\hline
$O(h^r),r=$  & 1.0245 &1.9886 &1.9889\\ \hline
\end{tabular}
\end{table}

\begin{table}[h]
\caption{Case 1. WG solutions and their convergence on triangular
elements.}\label{ex1_tri} \center
\begin{tabular}{|c||c|c|c|}
\hline
meshsize $h^{-1}$  & $\3bar Q_hu-u_h\3bar$ & $\|Q_hu-u_h\|$ & $\|Q_hu-u_h\|_{\mathcal{E}_h}$\\
\hline\hline
  4    &1.3567e+000  &1.5399e-001 &6.5585e-02\\ \hline
  8    &6.8946e-001  &3.9419e-002 &1.3106e-02\\ \hline
  16   &3.4613e-001  &9.9131e-003 &3.0102e-03\\ \hline
  32   &1.7324e-001  &2.4819e-003 &7.3455e-04\\ \hline
  64   &8.6641e-002  &6.2072e-004 &1.8249e-04\\ \hline
  128  &4.3323e-002  &1.5519e-004 &4.5550e-05\\ \hline\hline
$O(h^r),r=$  & 0.9949 &1.9925 &2.0855\\ \hline
\end{tabular}
\end{table}

Tables \ref{ex1_rect} and \ref{ex1_tri} show the rate of convergence
for the corresponding WG solutions in $H^1$ and $L^2$ norms on
rectangular and triangular meshes, respectively. The rectangular
mesh is constructed by uniformly partitioning the domain into
$n\times n$ sub-rectangles. The triangular mesh is obtained by
dividing each rectangular element into two triangles by the diagonal
line with a negative slope. The mesh size is denoted by $h=1/n$ for
both the rectangular and triangular meshes. The numerical results
indicate that the WG solution with $k=1$ is convergent with rate
$O(h)$ in $H^1$ and $O(h^2)$ in $L^2$ norms.

\subsection{Case 2: Degenerate Elliptic Problems} The second testing problem
is defined in the square domain $\Omega=(0,1)^2$ for the following
second order partial differential equation
$$
-\nabla\cdot(a \nabla u)=f,\qquad a=xy.
$$
Note that the coefficient $a=xy\ge 0$ in the domain and vanishes at
the origin. The PDE under consideration is thus elliptic, but with
some degeneracy near the origin. The WG finite element method
(\ref{wg}) is still applicable, and the corresponding discrete
problem admits a unique solution. However, the convergence theory
established in previous sections for the WG finite element method
can not be applied without any modification.

In our numerical tests, the exact solution is given by
$u=x(1-x)y(1-y)$, which corresponds to a homogeneous Dirichlet
boundary condition. Like the case 1, the function $f=f(x,y)$ is
given to match the exact solution.

\begin{table}[h!]
\caption{Comparison of Convergence for Three Finite Element Schemes
for a Degenerate Elliptic Problem} \label{tab:wg-dg}

\begin{tabular}{||c||cc||cc||cc||}

\multicolumn{1}{c}{\; }  & \multicolumn{2}{c}{WG-FEM $P_1P_1$} & \multicolumn{2}{c}{DG $P_1$} & \multicolumn{2}{c}{WG-FEM $P_0P_0$}\\
    \hline
   $\begin{matrix} \mbox{meshsize} \\
   h^{-1}\end{matrix} $& $H^1$-error & $L^2$-error & $H^1$-error & $L^2$-error & $H^1$-error & $L^2$-error\\
    \hline\hline

${8}$ &  2.51e-02 &  1.46e-03 & 3.98e-02 &  6.29e-03 & 5.16e-02 &  2.01e-03\\ \hline
 ${16}$   &  1.26e-02 &  3.74e-04 & 3.01e-02 &
2.92e-03 & 3.98e-02 &  9.30e-04\\ \hline
 ${32}$   &  6.31e-03 &  9.47e-05 &  2.23e-02 &  1.32e-03 & 2.96e-02 &  4.02e-04\\ \hline
 ${64}$   &  3.16e-03 &  2.39e-05 &  1.62e-02  & 5.89e-04 & 2.15e-02 &  1.70e-04\\ \hline
 ${128}$   &  1.58e-03 &  6.04e-06 &  1.17e-02  & 2.64e-04 &  1.55e-02  & 7.17e-05\\ \hline \hline
   $\begin{matrix}O(h^r)\\r=\end{matrix}$ & 9.97e-01 &   1.98e+00 &4.42e-01 &   1.15e+00 & 4.36e-01 &   1.21e+00
 \\ \hline
   \end{tabular}
\end{table}

In Table \ref{tab:wg-dg}, the column corresponding to WG-FEM
$P_1P_1$ refers to the computational results obtained from the
numerical scheme (\ref{wg}) with piecewise linear functions on each
element and its edges. The column corresponding to DG $P_1$ is the
result arising from the interior penalty method with piecewise
linear functions. The last column corresponds to results from the
weak Galerkin method detailed in \cite{wy} with piecewise constants.
These three methods were chosen for comparison because they have the
same rate of convergence in theory when the error is measured
between the finite element solution and a certain interpolation of
the exact solution.

The computational results indicate that the new WG-FEM scheme
(\ref{wg}) presented and analyzed in the present paper has optimal
order of convergence in both $H^1$ and $L^2$, while the other two
converges with significantly lower orders. The $H^1$ norm in the
table refers to discrete equivalence for each respective scheme.

\subsection{Case 3: WG-FEM on Deformed Rectangular Meshes} We solve the same problem as in
Case 1 on deformed rectangular meshes. We start with an initial
deformed rectangular mesh, shown as in Figure \ref{Deform_Mesh}
(Left). The mesh is then successively refined by connecting the
barycenter of each (coarse) element with the middle points of its
edges, as shown in the dotted line in Figure \ref{Deform_Mesh}
(Right). The numerical results are presented in Table \ref{ex3_rec},
which show an optimal order of convergence in various norms.

\begin{figure}[!h]
\centering
\begin{tabular}{cc}
  \resizebox{2.45in}{2.1in}{\includegraphics{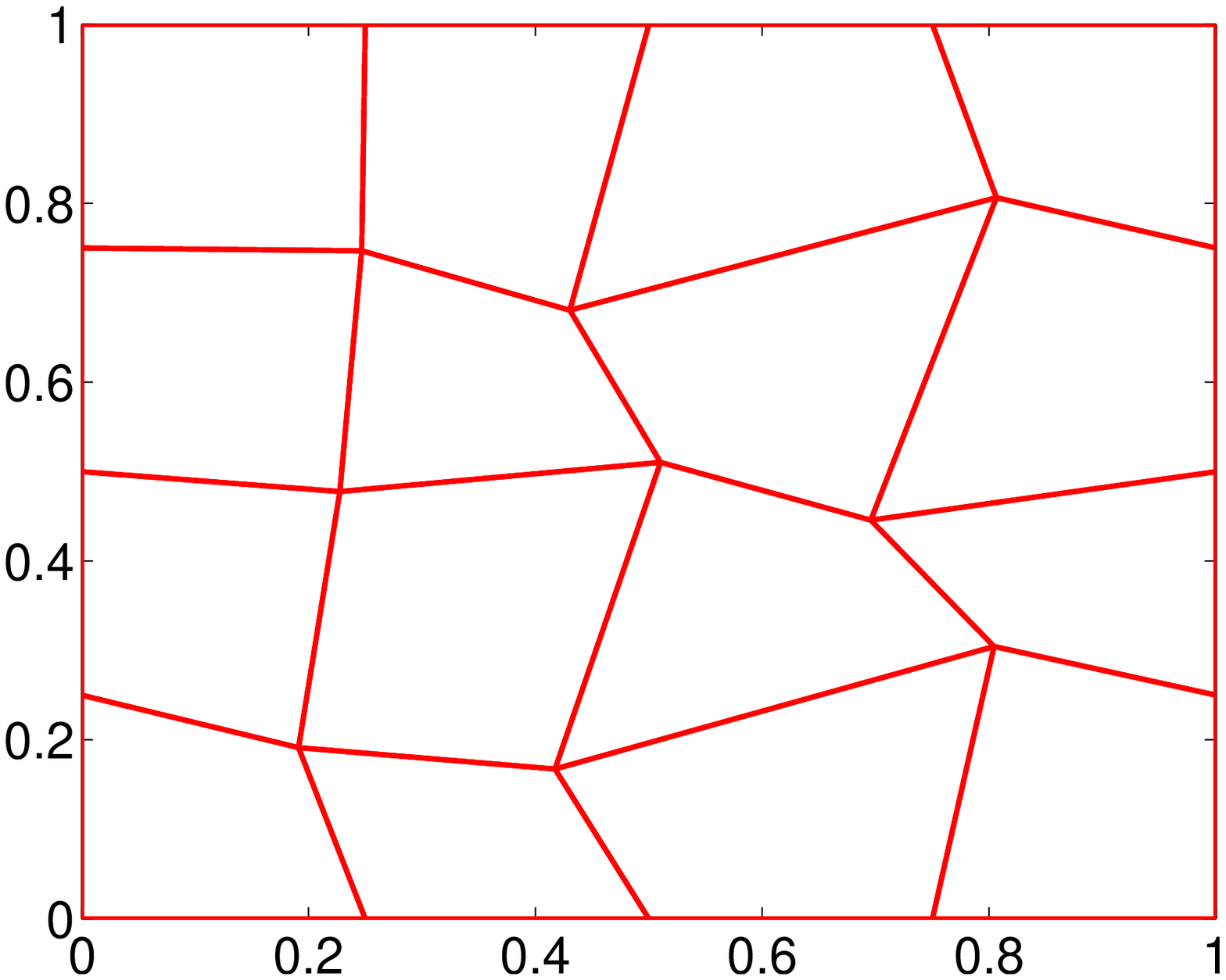}}
  \resizebox{2.45in}{2.1in}{\includegraphics{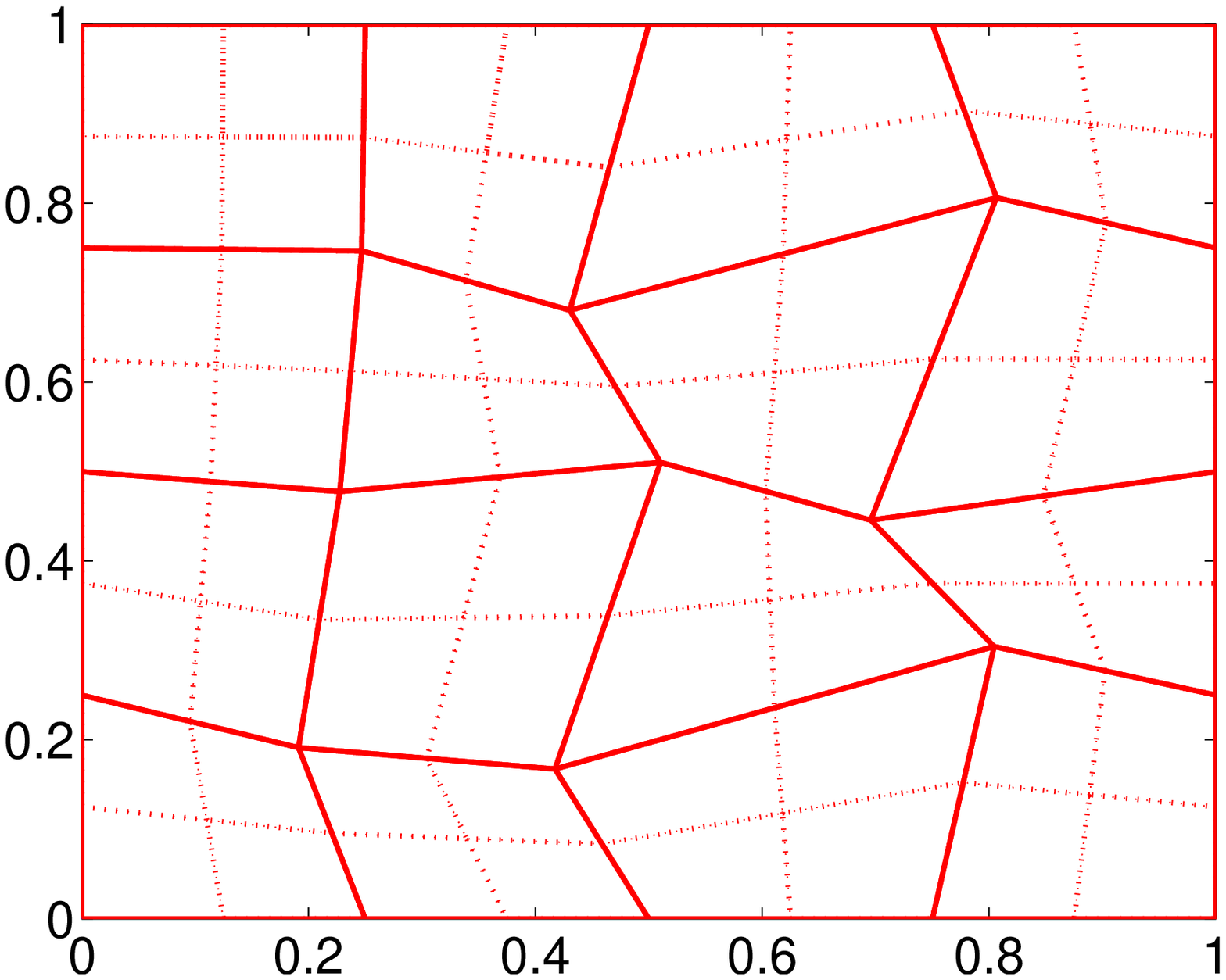}}
\end{tabular}
\caption{Case 3: An initial mesh (Left) and its refinement (Right).
}\label{Deform_Mesh}
\end{figure}

\begin{table}[h]
\caption{Case 3. Convergence rate on deformed
rectangles.}\label{ex3_rec} \center
\begin{tabular}{|c||c|c|c|}
\hline
meshsize {$h$} & $\3bar Q_hu-u_h\3bar$ & $\|Q_hu-u_h\|$ & $\|Q_hu-u_h\|_{\mathcal{E}_h}$\\
\hline\hline
   2.8790e-01  & 2.3056e+00  & 3.0235e-01 &8.2633e-02\\ \hline
   1.4395e-01  & 1.1673e+00  & 7.8108e-02 &2.1396e-02\\ \hline
   7.1974e-02  & 5.8473e-01  & 1.9652e-02 &5.3912e-03\\ \hline
   3.5987e-02  & 2.9241e-01  & 4.9203e-03 &1.3503e-03\\ \hline
   1.7993e-02  & 1.4619e-01  & 1.2445e-03 &3.3774e-04\\ \hline
   8.9967e-03  & 7.3095e-02  & 3.1112e-04 &8.4445e-05\\ \hline \hline
$O(h^r),r=$    & 0.9828      &1.9618 &1.9893\\ \hline
\end{tabular}

\end{table}

\subsection{Case 4: WG-FEM on Meshes with Hanging Nodes} We solve the same problem as in
Case 1 on deformed rectangular meshes with hanging nodes in the
finite element partition. The initial mesh is shown as in Figure
\ref{Deform_Mesh2} (Left). The mesh on the right in Figure
\ref{Deform_Mesh2} is generated by following the same uniform
refinement procedure as described in Case 3. It should be pointed
out that the initial mesh has a hanging node in the usual
definition.

For the finite element partition $\T_h$ with hanging nodes, the WG
finite element method (\ref{wg}) must be modified as follows. For
edge containing hanging nodes, the edge shall be further partitioned
into smaller segments by using the hanging nodes. Then the
corresponding finite element space defined on this edge will be
piecewise linear functions with respect to the new partition; the
finite element space on each element remains unchanged. For example,
in Figure \ref{Mesh_HangNode}, the elements $K_1$, $K_2$, and $K_3$
share one hanging node $M$. In the WG finite element method, the
edge $AB$ needs to be divided into two pieces: $AM$ and $MB$. The
corresponding finite element function $v_b$ is taken as a piecewise
linear function on edges $AM$ and $MB$. We point out that the
refinement method adopted here may produce elements around the
hanging node which are not shape regular as defined in Section 3.
The numerical results are presented in Table \ref{ex4_rec}. Readers
are encouraged to draw conclusions from this table.

\begin{table}[h!]
\caption{Case 4: WG solutions and their convergence on deformed
rectangular elements with hanging nodes.}\label{ex4_rec} \center
\begin{tabular}{|c||c|c|c|}
\hline
meshsize {$h$} & $\3bar Q_hu-u_h\3bar$ & $\|Q_hu-u_h\|$ & $\|Q_hu-u_h\|_{\mathcal{E}_h}$\\
\hline\hline
   4.2512e-01  & 3.8064e+00 &  9.1677e-01 &2.5509e-01\\ \hline
   2.1256e-01  & 2.2593e+00 &  3.2970e-01 &7.3355e-02 \\ \hline
   1.0628e-01  & 1.2308e+00 &  9.2302e-02 &2.2037e-02\\ \hline
   5.3140e-02  & 6.3470e-01 &  2.3300e-02 &6.1413e-03\\ \hline
   2.6570e-02  & 3.2104e-01 &  5.8376e-03 &1.7649e-03\\ \hline
   1.3285e-02  & 1.6129e-01 &  1.7094e-03 &5.1822e-04\\ \hline \hline
$O(h^r),r=$    & 0.9201      &1.8508 &1.7912\\ \hline
\end{tabular}
\end{table}

\begin{figure}[!h]
\centering
\begin{tabular}{cc}
  \resizebox{2.45in}{2.15in}{\includegraphics{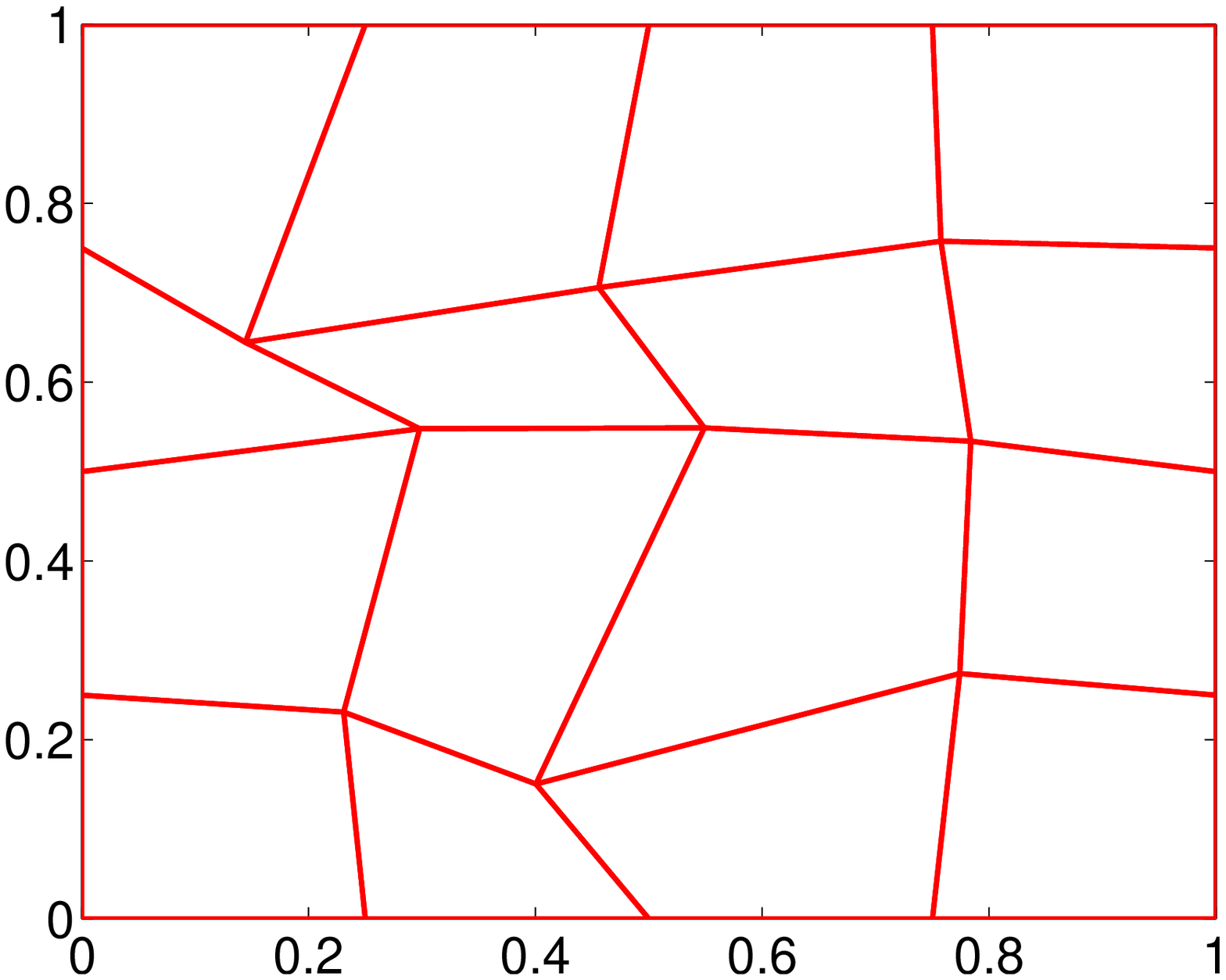}}
  \resizebox{2.45in}{2.15in}{\includegraphics{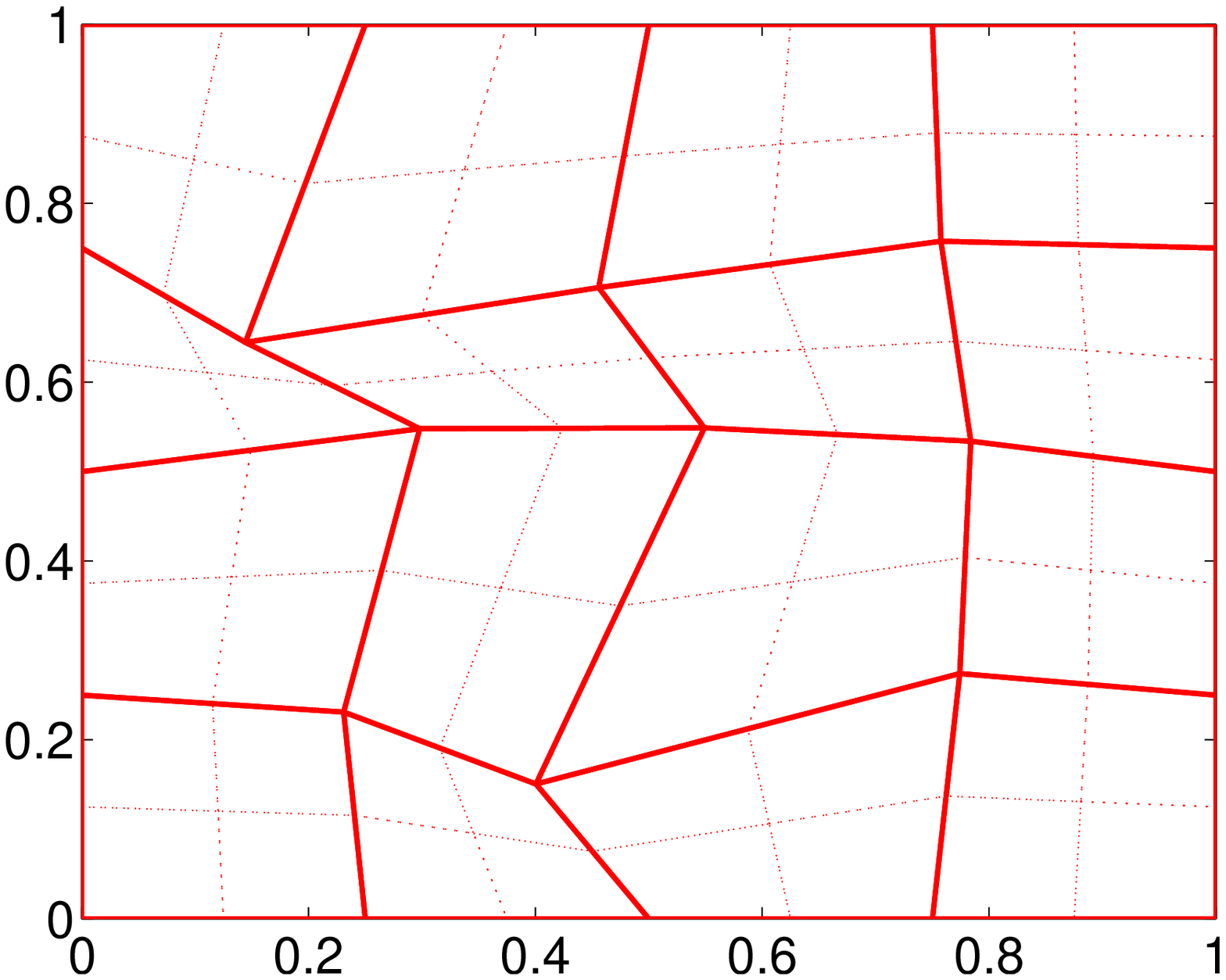}}
\end{tabular}
\caption{Case 4: Mesh level 1 (Left) and Mesh level 2(Right).
}\label{Deform_Mesh2}
\end{figure}

\begin{figure}[!h]
\centering
\begin{tabular}{c}
  \resizebox{2.45in}{2.0in}{\includegraphics{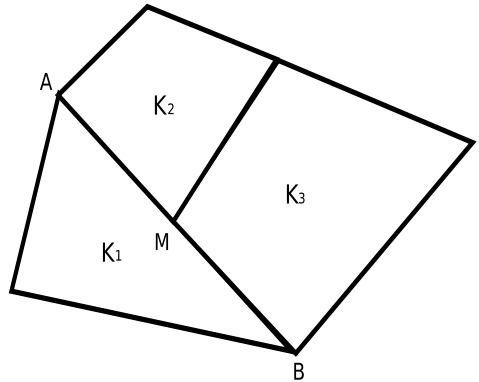}}
  \end{tabular}
\caption{Case 4: Elements around a hanging node in the mesh.
}\label{Mesh_HangNode}
\end{figure}

\vfill\eject

\end{document}